\documentclass[a4paper,11pt,reqno]{article}

\usepackage{hyperref}       
\usepackage{url}            
\usepackage{amsthm}
\usepackage{amsfonts}
\usepackage{amsmath}
\usepackage{amssymb}
\usepackage{color}
\usepackage{multirow}
\usepackage{float}
\usepackage{cite}
\usepackage{mathtools}
\usepackage{blkarray}
\usepackage{booktabs}
\usepackage{hhline}
\usepackage{float}
\usepackage{caption}
\usepackage{longtable}
\usepackage{array}
\usepackage[margin=2.7cm]{geometry}
\usepackage{enumerate}
\newcolumntype{C}[1]{>{\centering\arraybackslash}m{#1}}
\usepackage[normalem]{ulem}
\usepackage{comment}
\usepackage{fancyvrb}
\usepackage{titlesec}
\usepackage{sectsty}
\usepackage{url}

\theoremstyle{definition}

\newtheorem{theorem}{Theorem}[section]
\newtheorem{corollary}[theorem]{Corollary}
\newtheorem{lemma}[theorem]{Lemma}
\newtheorem{definition}[theorem]{Definition}
\newtheorem{proposition}[theorem]{Proposition}
\newtheorem{remark}[theorem]{Remark}

\newtheorem{example}[theorem]{Example}

\usepackage{tikz, framed}
\usetikzlibrary{arrows,shapes,matrix,decorations.pathmorphing}
\usetikzlibrary{backgrounds}

\newcommand{\Z}				{\mathbb{Z}}

\newcommand{\R}				{\mathbb{R}}

\newcommand{\F}				{\mathbb{F}}

\newcommand{\Id}            {\mathrm{Id}}

\newcommand{\CurA}			{\mathcal{A}}
\newcommand{\CurB}			{\mathcal{B}}

\newcommand{\CurL}			{\mathcal{L}}

\newcommand{\CurO}			{\mathcal{O}}
\newcommand{\CurP}			{\mathcal{P}}

\newcommand{\bP}             {\textbf{P}}

\DeclareMathOperator{\Aut}		{Aut}

\makeindex
\date{}
\begin{document}

%\begin{thebibliography}{10}

\title{Successive Minima, Determinant and Automorphism Groups of Hyperelliptic Function Field Lattices}
\author{Lilian Menn\thanks{Master Study at the Institute of Mathematics, University of Zürich, email: lilian.menn@uzh.ch} \and Elif Sa\c{c}\i kara\thanks{IMSV, University of Bern, email: elif.sacikarakariksiz@unibe.ch}}

\maketitle

\begin{abstract}
In this paper, we contribute to previously known results on lattices constructed by algebraic function fields, or function field lattices in short. First, motivated by the non-well-roundedness property of certain hyperelliptic function field lattices \cite{atecs2016note}, we explore the successive minima of these lattices in detail. We also study the determinant of hyperelliptic function field lattices. Finally, we show a connection between the automorphism groups of algebraic function fields and function field lattices, based on ideas from \cite{bottcher2016lattices}.
\end{abstract}


\section{Introduction}

The idea of constructing lattices from function fields was introduced by Rosenbloom and Tsfasman \cite{rosenbloom1990multiplicative} around 1990, and independently by Quebbemann \cite{quebbemann1989lattices}, to create families of asymptotically good lattices. These lattices are structured from subgroups of the group of principal divisors of a given function field, with the lattice vectors corresponding to elements of the function field having specific zeros and poles.

Over the course of the two decades following their construction 
almost no research was done on these lattices. Starting in $2014$ function field lattices were investigated in  
\cite{fukshansky2014lattices}. In this paper, the lattices from elliptic function fields were studied, including properties such as their determinant, kissing number, covering 
radius and the fact that they are well-rounded. Adding on to this, 
in \cite{sha2015lattices} these lattices were shown to have 
a basis made up of minimal vectors.  
Lattices from Hermitian function fields and many of their properties, including their automorphisms, were studied in \cite{bottcher2016lattices}.
A notable distinction was made in, \cite{atecs2016note}, where it was shown that certain hyperelliptic function field lattices are not well-rounded.

Firstly, this paper builds on these previous works by investigating the successive minima of hyperelliptic function field lattices that are not well-rounded. Specifically, we analyze these lattices in terms of their genus, defining polynomial, and the sets of places used in their construction. A significant challenge lies in understanding principal divisors with support on specific sets of places.

We also compute the determinant of certain hyperelliptic function field lattices, a topic previously unexplored. Our findings suggest that, under certain assumptions, some hyperelliptic function field lattices have higher densities than others. For lattices constructed from rational places, even determining a basis is difficult, partly due to the complex structure of the zero class group. This complexity makes the precise computation of the determinant more difficult to study.

Finally, much of this paper is dedicated to studying the automorphism groups of function field lattices. We aim to understand whether automorphisms of function fields induce automorphisms of the lattices constructed from them. While the automorphism groups of function fields and lattices are generally not equivalent, they are related. As an example, we show that automorphisms of the rational function field induce a subgroup of the automorphism group of the corresponding rational function field lattice. We conduct a similar analysis for elliptic and hyperelliptic function field lattices.

To study these automorphisms, we use two approaches: (1) analyzing automorphisms arising from the function field itself and (2) identifying additional automorphisms not directly related to the function field. The automorphisms we discover typically act as permutations of the coordinates of the function field lattice. 

The structure of this paper is as follows. In Preliminaries, we briefly review lattice constructions from algebraic function fields and previously known results. Next, we explore the successive minima of certain non-well-rounded function field lattices. In Section \ref{sec:Determinant}, we compute the determinant of some hyperelliptic function field lattices, providing the necessary background from algebraic function fields. Finally, in Section \ref{sec:Automorphisms}, we explore the connection between the automorphism groups of algebraic function fields and function field lattices, providing known results about the automorphism groups of both algebraic function fields and lattices. For the sake of completeness, we present the previously known results on lattices from finite abelian groups and Hermitian function field lattices in separate subsections.

\section{Preliminaries}\label{sec:preliminaries}

In this section, we introduce the construction and structure of function field lattices. For consistency, we use the notation and introduction given in
\cite{atecs2017lattices, lilymaster}.\\
Let $F/\mathbb{F}_q$ be an algebraic function field and let a set $\mathcal{P}$ denote $n+1$ distinct places;
\begin{align*}
    \mathcal{P} = 
\{ P_0, \dots, P_n\} \subset \mathbb{P_F}.
\end{align*}

Then correspondingly, we define the set $\mathcal{O}_{\mathcal{P}}^* \coloneqq 
\{ z \in F \: | \: \mathrm{supp}(z) \subseteq \mathcal{P} \}$.
It is easy to see that this set forms an abelian group with respect to field multiplication. In the literature, $\mathcal{P}$ generally consists of only rational places, but this is not necessary for the construction.

\begin{definition}\label{defn:FuncFieldLattice}
    Given an $(n+1)$-tuple $\mathcal{P}$ of places of a 
    function field $F/\mathbb{F}_q$, the associated \textbf{function 
    field lattice} $\mathcal{L}_\mathcal{P}$ is defined as the image of the map 
    \begin{align}\label{ConstructionFFLatticeMap}
       \Phi_\mathcal{P} : \mathcal{O}_{\mathcal{P}}^* 
    &\rightarrow \mathbb{R}^n \nonumber\\
 z &\mapsto (v_{P_0}(z)\cdot \mathrm{deg}(P_0), 
    \dots , v_{P_n}(z) \cdot \mathrm{deg}(P_n) )\nonumber. 
    \end{align}
\end{definition}

We note that the map $\Phi_\mathcal{P}$ is a group homomorphism. Indeed, for any two elements $z_1, z_2$ in $\mathcal{O}_{\mathcal{P}}^*$ we have 
$\Phi_\mathcal{P}(z_1 \cdot z_2 ) = \Phi_\mathcal{P}(z_1) + \Phi_\mathcal{P}(z_2)$. This becomes more obvious when 
comparing the lattice to the additive group of principal divisors with support in $\mathcal{P}$. 
It can be easily shown that a function field lattice $\mathcal{L}_\mathcal{P}$ has rank $n$ and is a sublattice of the root lattice $\CurA_n= \{ v \in \Z^{n+1} 
    \: \mid \: \sum^{n}_{i=0} v_i = 0 \}$ \cite{atecs2016note}.

The following propositions list what we know about the length of the vectors and the determinant of function field lattices. These are important to show properties like well-roundness or to study the determinant, for example.

\begin{proposition}\cite[Proposition 2.2.1]
{atecs2017lattices}\label{ff lat norm}
Let $\CurP$ and $\CurO_{\CurP}^*$ be as in Definition \ref{defn:FuncFieldLattice}, then the following holds.
    \begin{enumerate}
        \item[(i)] Let $z \in \CurO_\CurP^*$, then $||\Phi_\CurP(z)||^2 \equiv 
        0 \:\mathrm{mod}\: 2$, which means $\CurL_\CurP$ is an even lattice.
        \item[(ii)] Let $z \in \CurO_\CurP^*$, then 
        $||\Phi_\CurP(z)|| \geq \sqrt{2\mathrm{deg}(z)}$. Moreover, equality 
        holds if and only if the pole and zero of $z$ completely split 
        in the extension $F/\F_q(z)$.
        \item[(iii)] The minimum distance of $\CurL_\CurP$ satisfies 
        $d(\CurL_\CurP) \geq \sqrt{2\gamma}$, where $\gamma$ is the 
        gonality of $F$.
    \end{enumerate}
\end{proposition}

\begin{proposition}\cite[Proposition 2.2.2]
{atecs2017lattices}\label{prop:ff_lat_det}
Let $\mathcal{L}_\mathcal{P}$ be the function field lattice with  $\mathcal{P} = (P_0, P_1, \ldots, P_n)$. Let $d_i$ be the degrees of the places $P_i$, and let $d = \mathrm{gcd}(d_0, \ldots, d_n)$.
    \begin{enumerate}
        \item[(i)] The index $[\CurA_n:\mathcal{L}_\mathcal{P}]$ is given by 
        $$[\CurA_n:\mathcal{L}_\mathcal{P}] = \frac{d_0d_1\cdots d_n}{d} \cdot h_0$$
        where $h_0$ is an integer which 
        divides the class number $h$.

        \item[(ii)] The determinant of the function field lattice 
        $\mathcal{L}_\mathcal{P}$ is given 
        by $$\mathrm{det}(\mathcal{L}_\mathcal{P}) = 
        \sqrt{n+1}\cdot \frac{d_0d_1\cdots d_n}{d} \cdot h_0.$$
    \end{enumerate}
\end{proposition}

The following examples elaborate the definition for well-know function field families and provide previously known results in the literature. %\textcolor{blue}{sentence needs work}

\begin{example}[Rational Function Field Lattices]\label{eg:RationalFuncFieldLattices}

 Let $\F_q(x)$ denote the rational function field and choose a tuple $\mathcal{P}$ of any $n+1$ rational places in $K(x)$, i.e. 
\begin{align*}
    \mathcal{P} = (P_0, P_1, ... , P_n),
\end{align*}
    where $P_i \in \mathbb{P}^1_{K(x)}$ for each $i=1,\ldots, n+1$.
    Therefore, by Definition \ref{defn:FuncFieldLattice}, the function field lattice $\mathcal{L}_\mathcal{P}$ is 
    equal to the root lattice $\CurA_n$. More explicitly, we first consider the divisors $P_i - P_{i+1}$ for 
    $0\leq i \leq n$. Since the class number of the rational function 
    field is $h=1$, these divisors are all principle, which means the 
    $n$ vectors of the form 
    $$ (0, \dots, 0, 1, -1, 0, \ldots , 0),$$
    where $1$ and $-1$ are in the $i$-th and $i+1$-th coordinates respectively,
    are all contained in $\mathcal{L}_\mathcal{P}$. These vectors generate $\CurA_n$. Since every function field lattice is a sublattice of the root lattice $\CurA_n$, we conclude that the lattices  $\mathcal{L}_\mathcal{P}$ and $\CurA_n$ are indeed equal. Therefore, it is well-rounded, its minimum distance is $\sqrt{2}$, and its determinant is $\sqrt{n+1}$. 
\end{example}

\begin{example}[Elliptic Function Field Lattices]\label{eg:EllipticFuncFieldLattices}
    
Let $E/\F_q$ be an elliptic function field, that is, an algebraic function field of genus 1. %\textcolor{blue}{"By fixing" doesnt' follow with rest of sentence} 
By fixing $\CurP$ to be the set of all rational places of $E/\mathbb{F}q$, i.e., $\CurP = \mathbb{P}^{1}(E)$, these lattices were studied by \cite{fukshansky2014lattices} and \cite{sha2015lattices}. More specifically, the former establishes the well-roundedness property of elliptic function field lattices, while the latter builds on this by showing that elliptic function field lattices have a basis of minimal vectors, focusing mainly on the group structure of elliptic curves. In particular, \cite{fukshansky2014lattices} shows that when $n+1 \geq 4$, the minimum distance of $\CurL_\CurP$ is $d(\CurL_\CurP) = 2$, and the minimal vectors correspond to principal divisors of the form $P + Q - R - S$, where $P, Q, R$, and $S$ are distinct places in $\CurP$. In \cite{sha2015lattices}, it is proven that when $\CurP$ is cyclic, the elliptic function field lattice $\CurL_\CurP$ is equal to the Barnes lattice $\CurB_n$, defined as\begin{align}%\label{DefinitionExample:Barnes}
        \CurB_n \coloneqq \{ v \in \CurA_n \: \mid \: \sum_{i=0}^{n} i v_i \equiv 
    0 \: \mathrm{mod} \: n+1 \}. \nonumber 
    \end{align} 
Furthermore, \cite{sha2015lattices} shows that for $n+1 \geq 5$, the elliptic function field lattice $\CurL_\CurP$ has a basis of minimal vectors. Even though this was already proven in the same paper, we can observe, as an immediate consequence of Proposition \ref{prop:ff_lat_det}, that for an elliptic function field lattice $\CurL_\CurP$ with $|\CurP| = n+1 \geq 2$, we have $\det(\CurL_\CurP) = (n+1)^{\frac{3}{2}}$ \cite{lilymaster} %\textcolor{blue}{do we need this? maybe reword}.
\end{example}

Since hyperelliptic function fields are more complex than rational and elliptic ones, we address their corresponding lattice structures in the following separate subsection.

\subsection{Hyperelliptic Function Field Lattices}\label{eg:HyperellipticFuncFieldLattices}
Here, we summarize all the results from \cite{atecs2017lattices} that are necessary for our work. Let $F = \mathbb{F}_q(x,u)$ be a hyperelliptic function field with defining equation 
\begin{align*}
    u^2 = f(x),
\end{align*}
where $f(x)$ is a square free polynomial. For the choice of $\CurP$ in hyperelliptic function fields, we have three possible cases: a degree one place can either remain ramified as a degree one place, split into two degree one places, or become inert as a degree two place.%\textcolor{blue}{explain idea a little more} 
With this in mind, rather than simply choosing all the rational places of $F$, we work primarily with the following specific set of places for $\CurP$.
\begin{align*}
    \CurP = (P_\infty, P_2, \ldots, P_r, Q_1, \ldots, Q_s), 
\end{align*}
where $P_\infty$ denotes the pole of $x$, and  $P_2, \ldots, P_r$ 
denote all of the other rational places which are ramified in $F$,
and $Q_1, \ldots, Q_s$ denote all inert places of degree $2$. 
Furthermore, in hyperelliptic function fields, the ramified places correspond one-to-one with the roots of the defining polynomial $f(x)$. This means that the number of ramified places of degree one in $F$ is $1 \leq r \leq 2g+2$ \cite{stichtenoth}.

To understand the generating vectors of the lattice $\CurL_\CurP$, one can first analyze the generating elements of the ring $\CurO_\CurP^*$. To this end, one fixes 
$\alpha_i, \beta_j\in \mathbb{F}_q$ for $2 \leq i \leq r, \: 1 \leq j \leq s$ such that $(x-\alpha_i)$ and $(x-\beta_j)$ are prime elements 
of $P_i$ and $Q_j$, respectively. Therefore, the group $\CurO_\CurP^*$ can be described as follows:%\textcolor{blue}{"in the case that" change}
    \begin{enumerate}
        \item[(i)] In case $\mathrm{char}(F) = 2$ or in case 
        $\mathrm{char}(F) \neq 2$ and $f(x)$ does not split into linear 
        factors,
        \begin{align*}
            \CurO_\CurP^* = \{ \delta \cdot \prod_{i=2}^r (x-\alpha_i)^{a_i} \prod_{j=1}^s (x-\beta_j)^{b_j} \: | \: 
            \delta \in \mathbb{F}_q^*,\: a_i, b_j \in \Z \}.
        \end{align*}

        \item[(ii)] In case $\mathrm{char}(F) \neq 2$ and $f(x)$ splits into linear factors, 
        \begin{align*}
            \CurO_\CurP^* = \{ \delta \cdot u^\mu \cdot \prod_{i=2}^r 
            (x-\alpha_i)^{a_i} \prod_{j=1}^s (x-\beta_j)^{b_j} \: | \: 
            \delta \in \mathbb{F}_q^*,\: \mu \in \{0,1\},\:  a_i, b_j \in \Z \},
        \end{align*}
        
    \end{enumerate}

This clarifies which elements generate the group $\CurO_\CurP^*$. We primarily distinguish between two cases: whether $f(x)$ splits into linear factors or not. The reason for this distinction is as follows. In both cases we have the generating elements 
$x-\alpha_i$ and $x-\beta_j$ whose principal divisors are 
$(x-\alpha_i)^F = 2P_i-2P_\infty$ and $(x-\beta_j)^F = Q_j-2P_\infty$. 
However, in the case where $f(x)$ splits into linear factors 
the principal divisor of $u$ from the defining equation 
\begin{align*}
    u^2 = f(x)
\end{align*}
has support on $\CurP$ and hence is an element of $\CurO_\CurP^*$. 
The principal divisor of $u$ is then given by 
\begin{align*}
    (u)^F = -(2g+1)P_\infty + P_2 + \ldots + P_{2g+2},
\end{align*}
which cannot be generated by the elements $x-\alpha_i$ and $x-\beta_j$. 

Therefore, immediately, we see that the hyperelliptic function field lattice $\CurL_\CurP$ can be generated as follows. 
    \begin{enumerate}
        \item[(i)] In case $\mathrm{char}(F) = 2$ or in case 
        $\mathrm{char}(F) \neq 2$ and $f(x)$ does not split into linear 
        factors,
        \begin{align*}
            \CurL_\CurP = \langle \{ \Phi_\CurP(x-\alpha_i), 
            \Phi_\CurP(x-\beta_j) \: | \: 2 \leq i \leq r, \: 
            1 \leq j \leq s \}\rangle.
        \end{align*}
        \item[(ii)] In case $\mathrm{char}(F) \neq 2$ and $f(x)$ splits into linear factors,
        \begin{align*}
            \CurL_\CurP = \langle \{ \Phi_\CurP(x-\alpha_i), 
            \Phi_\CurP(x-\beta_j) \: | \: 2 \leq i \leq r, \: 
            1 \leq j \leq s \} \cup \{ \Phi_\CurP(u) \} \rangle.
        \end{align*}
        
    \end{enumerate}

By this representation, the following main result on hyperelliptic function lattices is obtained in \cite{atecs2017lattices}. Assume $g\geq 3$ then the following hold:%\textcolor{blue}{double check the requirements}
    \begin{enumerate}
        \item[(i)] The minimum distance of $\CurL_\CurP$ is 
        $d(\CurL_\CurP) = \sqrt{8}$ and 
        the vectors of the form 
        \begin{align*}
            \Phi_\CurP(x-\alpha_i), 
            \Phi_\CurP(x-\beta_j) \: \text{ for } \: 2 \leq i \leq r, \: 
            1 \leq j \leq s 
        \end{align*}
        are minimal vectors.
        \item[(ii)] The lattice $\CurL_\CurP$ is well-rounded.
    \end{enumerate}
 %\begin{corollary}\label{cor:hyp_ff_lat_gen}
 %   The hyperelliptic function field lattice $\CurL_\CurP$ can be 
 %   generated as follows. 
 %   \begin{enumerate}
 %       \item[(i)] In case $\mathrm{char}(F) = 2$ or in case 
 %       $\mathrm{char}(F) \neq 2$ and $f(x)$ does not split into linear 
 %       factors,
 %       \begin{align*}
 %           \CurL_\CurP = \langle \{ \Phi_\CurP(x-\alpha_i), 
 %           \Phi_\CurP(x-\beta_j) \: | \: 2 \leq i \leq r, \: 
 %           1 \leq j \leq s \}\rangle.
 %       \end{align*}
 %       \item[(ii)] In case $\mathrm{char}(F) \neq 2$ and $f(x)$ splits into linear factors,
 %       \begin{align*}
 %           \CurL_\CurP = \langle \{ \Phi_\CurP(x-\alpha_i), 
 %           \Phi_\CurP(x-\beta_j) \: | \: 2 \leq i \leq r, \: 
 %           1 \leq j \leq s \} \cup \{ \Phi_\CurP(u) \} \rangle.
 %       \end{align*}
        
 %   \end{enumerate}
 %\end{corollary}

All the algebraic function field lattices shown so far are well-rounded. The following case shows that well-roundness is not 
a universal property of function field lattices. Indeed, suppose $\CurP$ contains all the extensions of at least two rational places of $\mathbb{F}_q(x)$ which split in the extension $F$. Then the hyperelliptic function field lattice $\CurL_\CurP$ is not well-rounded. In particular this shows that in most cases, if $\CurP$ is chosen to be the set of rational places in $F$, then the resulting hyperelliptic function field lattice is not well-rounded.
 
\section{Successive Minima}
In this section, we address the natural question of determining the successive minima of the non-well-rounded function field lattices constructed in \cite{atecs2016note}. To do so, we first examine all possible cases by choosing $\CurP$ as in Subsection \ref{eg:HyperellipticFuncFieldLattices}. The following theorem covers all cases, showing whether the lattices are well-rounded or exhibit stronger properties. Additionally, we provide some new results not mentioned in \cite{atecs2017lattices}.

\begin{theorem}\label{thm:Hyperll_ff_lat_basis_vs_well-rounded}
    Let $F/\mathbb{F}_q(x)$ be a hyperelliptic function field 
    with defining equation $u^2 = f(x)$ where $f(x)$ is a 
    square-free polynomial of degree $2g+1$. Let 
    $\CurP$ be the set of rational ramified places $P_i$'s and inert places $Q_j$'s of degree $2$, denoted by  $$\CurP = \{ P_\infty, P_2, \ldots, P_r, Q_1, \ldots, Q_s \}.$$ Then the following hold. 
    \begin{enumerate}
        \item[(i)] If $\mathrm{char}(F) = 2$ then 
        $\CurL_\CurP$ has a basis of minimal vectors.
        \item[(ii)] If $\mathrm{char}(F) \neq 2$ and $f(x)$ does not split into linear factors then 
        $\CurL_\CurP$ has a basis of minimal vectors.
        \item[(iii)] If $\mathrm{char}(F) \neq 2$, $f(x)$ splits into linear factors and the genus is $g=3$ then 
        $\CurL_\CurP$ has a basis of minimal vectors.
        \item[(iv)]  If $\mathrm{char}(F) \neq 2$, $f(x)$ splits into linear factors and the genus is $g>3$ then 
        $\CurL_\CurP$ is well-rounded.
        \item[(v)] If $\mathrm{char}(F) \neq 2$, $f(x)$ splits into linear factors and the genus is $g=2$ then
        $\CurL_\CurP$ is not well-rounded.
    \end{enumerate}
\end{theorem}

\begin{proof}
   Well-roundedness for cases $(i)-(iv)$ was already proven in \cite{atecs2017lattices}.
    
Now, let us focus on cases $(i)$ and $(ii)$. By Subsection \ref{eg:HyperellipticFuncFieldLattices}, the vectors $\Phi_\CurP(x-\alpha_i)$ and $\Phi_\CurP(x-\beta_j)$, for $2 \leq i \leq r$ and $1 \leq j \leq s$, generate the lattice. Moreover, they are minimal vectors. Since there are exactly $r+s-1$ of these vectors, which is the rank of the lattice $\CurL_\CurP$, these vectors form a basis of minimal vectors.

To show $(iii)$, suppose the polynomial $f(x)$ splits and hence $u$ is an element in $\CurO_\CurP^*$. As stated in Subsection \ref{eg:HyperellipticFuncFieldLattices} the vectors 
$\Phi_\CurP(x-\alpha_i), \Phi_\CurP(x-\beta_j) \: 
\text{ for } \: 2 \leq i \leq 2g+2, \: 1 \leq j \leq s$ and $\Phi_\CurP(u)$ generate the lattice.
Consider the vector 
    \begin{align*}
        v^* &= \Phi_\CurP(u) - \Phi_\CurP(x-\alpha_2) - \ldots 
        - \Phi_\CurP(x-\alpha_{g+1})\\
        &= (-1, -1, \ldots, -1, 1, \ldots , 1, 0, \ldots, 0)
    \end{align*}
    where the first $g+1$ coordinates are $-1$, the next $g+1$ 
    coordinates are $1$ and the last $s$ coordinates are $0$.
    This vector has length $||v^*|| = \sqrt{2g+2} = \sqrt{8}$, since 
    the genus is $g=3$, and is therefore also a minimal vector.
    Now consider the set of $2g+1+s$ vectors 
    \begin{align*}
        \{v^*, \Phi_\CurP(x-\alpha_2), \ldots, \Phi_\CurP(x-\alpha_{2g+1}), 
        \Phi_\CurP(x-\beta_1), \ldots, \Phi_\CurP(x-\beta_s) \}.
    \end{align*}
    These vectors are all minimal and the rank of the lattice is $2g+1+s$. 
    Furthermore, they generate both $\Phi_\CurP(u)$ and $\Phi_\CurP(x-\alpha_{2g+2})$ as follows:
    \begin{align*}
        \Phi_\CurP(u) &= v^* + \Phi_\CurP(x-\alpha_2) + \ldots 
        + \Phi_\CurP(x-\alpha_{g+1})\\
        \Phi_\CurP(x-\alpha_{2g+2}) &= 2\cdot v^* + \Phi_\CurP(x-\alpha_2) 
        + \ldots + \Phi_\CurP(x-\alpha_{g+1})\\
        &-\Phi_\CurP(x-\alpha_{g+2})
        - \ldots - \Phi_\CurP(x-\alpha_{2g+1}).
    \end{align*}
  This implies that these vectors generate the entire lattice. Therefore, they form a basis of minimal vectors for $\CurL_\CurP$.

   Lastly, we show the last point $(v)$.
    First we prove that the minimum distance of $\CurL_\CurP$ in this case, is 
    $d(\CurL_\CurP) = \lambda_1(\CurL_\CurP) = \sqrt{6}$.
    Recall that by Proposition \ref{ff lat norm} all function field lattices are even, meaning that for all $z \in \CurO_\CurP^*$ we have 
    $||\Phi_\CurP(z)|| = \sqrt{2t}$ for some integer $t \geq 1$. Now, assume there exists a $z$ in $\CurO_\CurP^*$ such that 
    $||\Phi_\CurP(z)|| < \sqrt{6}$. We will go through the cases to show that this is not possible.
    
    If $||\Phi_\CurP(z)|| = \sqrt{2}$ then without loss of generality the vector has the 
    form 
    \begin{align*}
        \pm (0, \ldots, 0, 1, -1, 0, \ldots, 0).
    \end{align*}
    This implies that $z$ has degree $1$ and the extension degree $F/\mathbb{F}_q(z)$ 
    is $[F:\mathbb{F}_q(z)] = 1 $, which means $F$ must be the rational function field. However, this yields a contradiction.
    
    If $||\Phi_\CurP(z)|| = \sqrt{4}$ then without loss of generality the vector has the 
    form 
    \begin{align*}
        \pm (0, \ldots, 0, 1, 1, -1, -1, 0, \ldots, 0).
    \end{align*}
    Then, $z$ has degree 2, so $[F:\mathbb{F}_q(z)] = 2$. Since $\mathbb{F}_q(x)$ is the only rational subfield of degree 2, this implies that $\mathbb{F}_q(x) = \mathbb{F}_q(z)$. 
    If either the pole or zero of $z$ splits in $F$ then $z$ cannot be an element of $\CurO_\CurP^*$ due to how 
    $\CurP$ is defined. This means the pole and zero of $z$ must be either ramified or inert, 
    in which case $\Phi_\CurP(z)$ takes the form of
    \begin{align*}
        \pm (0, \ldots, 0, 2, -2, 0, \ldots, 0).
    \end{align*}
    By the same argument, we can exclude any $z$ such that the vector $\Phi_\CurP(z)$ has length $\sqrt{6}$ and is of the form 
    \begin{align*}
        \pm (0, \ldots, 0, 2, -1, -1, 0, \ldots, 0),
    \end{align*}
    since in this case $z$ would also have degree $2$. 

    Now consider the five vectors generated by $\Phi_\CurP(u)$ and 
    $\Phi_\CurP(x-\alpha_i)$ for $2\leq i \leq 6$, that is,
    \begin{align*}
    ( 1 , 1 , 1 , -1 , -1 , -1 , 0 ,\ldots , 0),\\
    ( 1 , 1 , -1 , 1 , -1 , -1 , 0 ,\ldots , 0),\\
    ( 1 , -1 , 1 , 1 , -1 , -1 , 0 ,\ldots , 0),\\
    ( -1 , 1 , 1 , 1 , -1 , -1 , 0 ,\ldots , 0),\\
    ( -1 , -1 , 1 , 1 , 1 , -1 , 0 ,\ldots , 0).
    \end{align*}
   One can verify that these vectors are linearly independent and all have length $\sqrt{6}$. We now need to show that there are no additional vectors of length $\sqrt{6}$ that are linearly independent from these.
    Recall that each place $Q_j$ for $1\leq j \leq s$ in $\CurP$ is 
    inert of degree $2$. This means the $(6+j)$-th entry of any vector $v$ in $\CurL_\CurP$ is a multiple of two. 
    Since we already eliminated all $z$ where $\Phi_\CurP(z)$ has length $\sqrt{6}$ with an entry equal to $\pm2$. This implies that any such element $z$ would have support only on the ramified places in $\CurP$.
    Since we are limited to six coordinates and must be sure that the sum of the entries is zero, there can be at most five linearly independent vectors that satisfy this condition. This implies that there are at most five linearly independent minimal vectors in $\CurL_\CurP$, meaning the lattice is not well-rounded.

\end{proof}

So far, we have not discussed the successive minima of function field lattices in detail because, for well-rounded lattices, all the successive minima are simply equal to the minimum distance. However, as shown in the previous theorem, there is a specific case where determining the successive minima becomes interesting. The following theorem addresses this case.

\begin{theorem}
    Let $\CurP$ be as in Theorem \ref{thm:Hyperll_ff_lat_basis_vs_well-rounded}.
    Assume $\mathrm{char}(F) \neq 2$, $f(x)$ splits into linear factors and the genus is $g=2$. Then the successive minima of $\CurL_\CurP$ 
    are 
    \begin{align*}
        \lambda_1(\CurL_\CurP) = \ldots = \lambda_5(\CurL_\CurP) = \sqrt{6}, \\ 
        \lambda_6(\CurL_\CurP) = \ldots = \lambda_{6+s-1}(\CurL_\CurP) = \sqrt{8}.
    \end{align*}

\end{theorem}

\begin{proof}
    We note that the proof of part $(v)$ of Theorem \ref{thm:Hyperll_ff_lat_basis_vs_well-rounded} shows that the first five successive minima are 
    \begin{align*}
        \lambda_1(\CurL_\CurP) = \ldots = \lambda_5(\CurL_\CurP) = \sqrt{6}.
    \end{align*}

   The final step is straightforward. Since the lattice is even, the next possible shortest length of a lattice vector after $\sqrt{6}$ is $\sqrt{8}$. 
    We provide the remaining $s$ vectors in 
    the lattice $\CurL_\CurP$ which are linearly independent from the five we already showed in the proof of part $(v)$ of Theorem \ref{thm:Hyperll_ff_lat_basis_vs_well-rounded}. These are the vectors $\Phi_\CurP(x-\beta_j)$ 
    for $1\leq j \leq s$, or more explicitly,
\begin{equation*}
    \begin{matrix}
        (-2, & 0, & 0, & 0, & 0, & 0, & 2, & 0, & \ldots, & 0),\\
        (-2, & 0, & 0, & 0, & 0, & 0, & 0, & 2, & \ldots, & 0),\\
        \phantom{(-2,} & \phantom{0,} & \phantom{0,} & \phantom{0,} & \vdots & \phantom{0,} & \phantom{0,} & \phantom{0,} & \phantom{\ldots,} & \phantom{0)}\\
        (-2, & 0, & 0, & 0, & 0, & 0, & 0, & 0, & \ldots, & 2).
    \end{matrix}
\end{equation*}
    These all have length $\sqrt{8}$ and are clearly linearly independent from the first minimal vectors and to each other. Therefore, the last $s$ successive minima are 
    \begin{align*}
        \lambda_6(\CurL_\CurP) = \ldots = \lambda_{6+s-1}(\CurL_\CurP) = \sqrt{8}.
    \end{align*}
    
\end{proof}

Now that we have thoroughly discussed the cases where $\CurP$ consists 
of ramified and inert places. Let us explore what happens when 
$\CurP$ contains only the rational places of $F$ and the genus is $g\geq 3$. 
As we mention in Subsection \ref{eg:HyperellipticFuncFieldLattices}, when 
$\CurP$ contains all extensions of at least two rational places that split, the lattice is not well-rounded. 
So, naturally we are interested in what the shortest vectors in these lattices look like. To this end, we first fix some notation. Let 
\begin{align*}
    \CurP = (P_1, P_2, \ldots, P_r, S_{11}, S_{12}, S_{21}, S_{22}\ldots, 
    S_{t1}, S_{t2}),
\end{align*}
where $P_i$ for $1\leq i \leq r$, as before, represent the rational 
ramified places, and $S_{j1}, S_{j2}$ for $1\leq j \leq t$ represent the 
two extensions of a rational place $S_j$ in $\mathbb{F}_q(x)$ which splits completely in $F$. The resulting lattice $\CurL_\CurP$ then has 
rank $r+2t-1$.

By Proposition \ref{ff lat norm} the minimum distance satisfies 
$d(\CurL_\CurP) = 2$. In \cite{atecs2017lattices} 
it is detailed that the vectors that reach this minimum are those 
which are associated to principal divisors of the form 
\begin{align*}
    S_{i1} + S_{i2} - S_{j1} - S_{j2},
\end{align*}
for $1\leq i\neq j \leq t$. There are at most $t-1$ linearly independent vectors of this form in the lattice $\CurL_\CurP$. 

These vectors all have length $2$ which means that the next shortest vectors would have length $\sqrt{6}$. The only principal divisors whose associated vectors have length $\sqrt{6}$ are of the from 
\begin{align*}
    \pm (2P_i - S_{j1} - S_{j2}),
\end{align*}
for $1\leq i \leq r$ and $1\leq j \leq t$. We can choose $r$ of these 
vectors which are linearly independent to the first $t-1$ minimum vectors discussed.% \textcolor{blue}{to the first t-1 (min) vectors discussed}.

Now we are missing $t$ vectors to reach the rank of the lattice. 
Notice how in each case so far the two extensions $S_{j1}, S_{j2}$ of 
a splitting place $S_j$ always appear together, with the same 
valuation.
This can be explained by the following fact.

\begin{lemma}\cite[Propopsition 4.4.11]{atecs2017lattices}
    Let $F/\mathbb{F}_q$ be a hyperelliptic function field with 
    genus $g$ and let 
    $z\in \CurO^*_\CurP$. If the length of the vector $||\Phi_\CurP(z)||$ satisfies 
    \begin{align*}
        ||\Phi_\CurP(z)|| < \sqrt{2g+2}
    \end{align*}
    then $z$ belongs to $\mathbb{F}_q(x)$.
\end{lemma}

This lemma particularly implies that any element in $\CurO^*_\CurP$ 
that is not in the base field is mapped to a vector in the lattice 
with length greater or equal to $\sqrt{2g+2}$. Recall that the 
conorm of the place $S_j$ is $\mathrm{con}_{F/K(x)}(S_j) = 
S_{j1}+S_{j2}$. This means that elements in the base field $z \in \mathbb{F}_q(x)$, which by the above lemma are those with shorter length, will satisfy $v_{S_{j1}}(z) = v_{S_{j2}}(z)$.

It appears that the vectors linearly independent from the shortest ones we have already discussed are those with unequal valuations at the places $S_{j1}$ and $S_{j2}$ for $1 \leq j \leq t$. Therefore, these vectors are not in the base field and, according to the above lemma, have a length greater than or equal to $\sqrt{2g+2}$.

These observations lead us to the following theorem on the successive minima of $\CurL_\CurP$.

\begin{theorem}
    Let $F/\F_q$ be an algebraic function field of genus $g\geq 3$ and suppose that $\CurP$ consists of only rational places. Then the following hold.
    \begin{align*} 
    \begin{cases}
        \lambda_i(\CurL_\CurP) = 2, &\text{ for } 1\leq i \leq t-1,\\
        \lambda_i(\CurL_\CurP) = \sqrt{6}, &\text{ for } t\leq i \leq 
        t+r,\\
        \lambda_i(\CurL_\CurP) \geq \sqrt{2g+2},
        &\text{ for } t+r \leq i \leq 2t+r-1.
    \end{cases}
\end{align*}
\end{theorem}

\section{Determinant}\label{sec:Determinant}

Recall that the determinant of a function field lattice, as described in Proposition \ref{prop:ff_lat_det}, depends on the number $h_0$, which divides the class number $h$. Therefore, our focus is on determining $h_0$, which, as discussed, represents the number of classes in $\mathrm{Div}_\CurP^0(F)/ \mathrm{Princ}_\CurP(F)$. Here, $\mathrm{Div}_\CurP^0(F)$ and $\mathrm{Princ}_\CurP(F)$ refer to the degree-zero divisors and principal divisors whose support is a subset of $\CurP$, respectively. For rational and elliptic function field lattices, determining $h_0$ is straightforward, but it becomes more challenging in the hyperelliptic case. In this section, we consider the cases of hyperelliptic function fields where we were able to determine $h_0$.

To this end, let us first recall the definition of semi-reduced divisors, which play an important role in understanding $h_0$.
 \begin{definition}
\label{defn:reducedDivisors}
    Let $F=K(x,y)$ be a hyperelliptic function field over 
    an algebraically closed field $K$. A 
    degree zero divisor $D \in \mathrm{Div}^0(F)$ is called \textbf{reduced} if
    \begin{enumerate}
        \item $v_P(D) \geq 0 $ for all $ P \in \mathbb{P}(F) 
        \backslash \{ P_\infty\} $.
        \item If $P \cap K(x)$ is ramified in $F$ then 
        $v_P(D) \in \{ 0, 1\}$.
        \item If $P \cap K(x)$ splits in $F$ such that 
        $\mathrm{con}_{F/K(x)}(P) = P_1 + P_2$ then 
        $v_{P_1}(D) > 0 $ implies that $v_{P_2}(D) = 0 $.
        \item $\frac{1}{2} \mathrm{deg}(D) \leq g$.
    \end{enumerate}
    When only the first three conditions are fulfilled then the 
    divisor is called \textbf{semi-reduced}.
\end{definition}

 Note that in the above definition the inert places in $F$ are not mentioned. This is because the original definition assumes $K$ is algebraically closed. When the base field is not algebraically closed, for an inert place $P$, the valuation of any reduced divisor $D$ is $v_P(D) = 0$.

\begin{proposition}
    \label{prop:InertReducedDiv}
  Let $D$ be a reduced divisor then $v_P(D)=0$ whenever $P$ is inert.  
\end{proposition}
\begin{proof}
    Let $P'$ be an inert place of $F$ and $P$ be the place such that 
    $P=P'\cap K(x)$. 
    We consider the divisor $P-P_{\infty}$, which is a principal 
    divisor in $K(x)$. Then its conorm $P'-2P_{\infty}$ is also principal. Now assume there is a divisor class $[D]$ in  $\mathrm{Cl}^0(F)$ 
    with a reduced divisor representative $D$ such that $v_{P'}(D)>0$, then there is an equivalent divisor $D' = D-v_{P'}(D)(P-2P_\infty)$ with $v_P(D')=0$. Then $D'$ is still a reduced divisor. 
\end{proof}

\begin{theorem}{\cite{salvador2006topics}}\label{thm:UniquenessReducedDivHyper}
     Let $F=K(x,y)$ be a hyperelliptic function field with 
     Jacobian $\mathrm{Cl}^0(F)$. Then every class in $\mathrm{Cl}^0(F)$ 
     contains a unique reduced divisor $D$.
\end{theorem}

Theorem \ref{thm:UniquenessReducedDivHyper} allows us count the classes in the class zero group by counting the possible reduced divisors. While this is generally quite difficult, we can significantly narrow our search by considering only divisors with support on $\mathcal{P}$, as we are focused on the number of classes in $\mathrm{Div}_\mathcal{P}^0(F) / \mathrm{Princ}_\mathcal{P}(F)$. In our case, $\mathcal{P}$ consists only of ramified and inert places. This simplifies the counting process because inert places do not appear at all, and ramified places appear with multiplicity of at most one in any reduced divisor. The next goal is to verify whether each reduced divisor represents a class in $\mathrm{Div}^0_\mathcal{P}(F) / \mathrm{Princ}_\mathcal{P}(F)$ and vice versa. If this holds, we can determine $h_0$.

 \begin{lemma}\label{lemma:form_of_reduced_div_hyp_ff}
     Let $\mathcal{P}$ be the set of $r$ rational ramified places
     and $s$ inert places of degree 2.
     Then each class in $\mathrm{Div}_\CurP^0(F) / \mathrm{Princ}_\CurP(F)$ contains a semi-reduced 
     divisor of the form $$D = 
     \sum_{i=1}^{r-1} a_i P_i - (\sum_{i=1}^{r-1} a_i) P_\infty$$
     for $a_i \in \{0,1\}$.
\end{lemma}

\begin{proof}
    Let $[D]$ be a divisor class in $\mathrm{Div}_\mathcal{P}^0(F)$. Then by Proposition 
    \ref{prop:InertReducedDiv} there is a principal divisor 
    $D_1 \in \mathrm{Princ}_{\mathcal{P}}(F)$ such that 
    $D-D_1$ does not have support on any of the inert places in 
    $\mathcal{P}$. 
    Similarly, we know that the divisor $2P_i-2P_\infty$ in $F$ is 
    principal because it is the conorm of $P_i-P_\infty$ in $K$. 
    Then for $1\leq i \leq r-1$ let $a_i = v_{P_i}(D) 
    \in \mathbb{Z}$. This means the divisor $$D_2 = \sum_{i=1}^{r-1}  \lfloor \frac{a_i}{2} 
    \rfloor (2P_i-2P_\infty)$$ is principal and 
    $D-D_1-D_2$ is of the desired form. 
    Since $D_1$ and $D_2$ are principal, the class $[D] \in \mathrm{Div}_\CurP^0(F) / \mathrm{Princ}_\CurP(F)$ is equal to the class $[D-D_1-D_2]$. Further, this divisor 
    is semi-reduced because it only has support on ramified places where 
    the valuation is either $1$ or $0$. This completes the proof.
\end{proof}

This is the first step in counting the different classes. The important thing to 
check is whether any of these semi-reduced or reduced divisors are 
equivalent to each other in $\mathrm{Div}_\CurP^0(F) / \mathrm{Princ}_\CurP(F)$. This is addressed in the proof of the following 
theorem.

\begin{theorem}\label{thm:Hyp_ell_ff_lat_det_ram}
    Let $F = \F_q(x,y)$ be a hyperelliptic function field with 
    defining equation $y^2 = f(x)$ for a square free polynomial of degree 
    $\mathrm{deg}(f(x)) = 2g+1$.
    Let $\mathcal{P} = (P_\infty, P_2, \dots , P_r, Q_1, 
    \dots Q_s)$ where $P_i$ are all ramified places of 
    degree $1$ and $Q_j$ are all inert places of degree $2$.
    The determinant of the function field 
    lattice $\mathcal{L}_\mathcal{P}$ is 
    $\mathrm{det}(\mathcal{L}_\mathcal{P}) = 
    \sqrt{n+1}\cdot 2^s \cdot h_0$ where $h_0$ is 
    given by
    \begin{align*}
        \begin{cases}
            h_0 = 2^{r-1} &\text{ when } f(x) \text{ does not split into 
            linear factors },\\
            h_0 = 2^{2g}  &\text{ when } f(x) \text{ splits into 
            linear factors. } 
        \end{cases}
    \end{align*}

\end{theorem}

\begin{proof}
    
    By Lemma \ref{lemma:form_of_reduced_div_hyp_ff} we know that each 
    class in $\mathrm{Div}_\mathcal{P}^0(F)
    / \mathrm{Princ}_\mathcal{P}(F)$ contains a semi reduced divisor 
    of the form 
    \begin{align}\label{eq:semi-reduced_div_form}
        D = 
     \sum_{i=1}^{r-1} a_i P_i - (\sum_{i=1}^{r-1} a_i) P_\infty
    \end{align}
     for $a_i \in \{0,1\}$.

    Since there only two options for each $a_i$, this means there are 
    $2^{r-1}$ different divisors of this form.
    First, let us assume that $f(x)$ does not split into linear 
    factors. Let $D$ be one of these $2^{r-1}$ divisors. 
    If $D$ is a reduced divisor then by Theorem \ref{thm:UniquenessReducedDivHyper} any two reduced divisors are in distinct 
    classes in $\mathrm{Div}^0(F)
    / \mathrm{Princ}(F)$. Since $\mathrm{Princ}_{\mathcal{P}}(F)$ 
     lies in $\mathrm{Princ}(F)$, they must also be in distinct classes in the 
    quotient group $\mathrm{Div}^0_{\mathcal{P}}(F)
    / \mathrm{Princ}_{\mathcal{P}}(F)$.    This shows that two reduced divisors are in distinct classes. 
    
    Next we show the same for one reduced and one semi-reduced divisor.
    Now let $D$ be only semi-reduced, i.e. 
    $\frac{1}{2}\mathrm{deg}(D) > g$. Because this divisor is already 
    semi-reduced, finding the equivalent reduced divisor is straightforward.
    Let $(u)^F$ be the principal divisor of $u$ in $F$
    \begin{align*}
        (u)^F = P_1 + \ldots +P_{r-1}+R_1+\ldots+R_t-(2g+1)P_\infty,
    \end{align*}
    where $R_1, \ldots R_t$ are the non-rational ramified places of $F$ with $t\geq1$.
    Then the divisor $\Tilde{D}= -(D-(u)^F)$ is reduced because it has
    \begin{align*}
        \frac{1}{2}(\mathrm{deg}(-D+(y))) = 
    \frac{1}{2}(2g+1 -\mathrm{deg}(D)) \leq g
    \end{align*} 
    and only has support on ramified places with valuation either 
    $0$ or $1$. Because $(u)^F$ is principal, $\Tilde{D}$ is in 
    the same equivalence class as $D$ in $\mathrm{Div}^0(F)
    / \mathrm{Princ}(F)$.
    This reduced divisor is not equal to any of the 
    reduced divisors given in Equation \ref{eq:semi-reduced_div_form}, which means 
    that $D$ is also in its distinct class. 
    It follows that all $2^{r-1}$ divisors given in Equation \ref{eq:semi-reduced_div_form} are in distinct classes.
    
    Finally, assume that $f(x)$ splits into linear factors. As before, the reduced divisors are all in distinct classes. However, in this case, the principal divisor $(u)^F$ is an element of $\mathrm{Princ}_\mathcal{P}(F)$ since all ramified places are rational and therefore in $\CurP$. This means, by the same argument as before, that for any semi-reduced divisor given in Equation \ref{eq:semi-reduced_div_form} there is an equivalent reduced divisor given in Equation \ref{eq:semi-reduced_div_form}.
    Therefore, the number of distinct classes is equal to the number of reduced divisors with support on $\mathcal{P}$, of which there are 
    \begin{align*}
        \sum_{i=0}^{g} {2g+1 \choose i} = 2^{2g}.
    \end{align*}  
\end{proof}    

\begin{remark}
    In the case where the polynomial $f(x)$ does not split into linear 
    factors, we prove in Lemma \ref{lemma:hyp_latt_similar_to_A} that the function field lattice is equal to 
    $\mathcal{L}_\mathcal{P} = 2\cdot \CurA_n$. This means we can easily 
    confirm our result in Theorem \ref{thm:Hyp_ell_ff_lat_det_ram}, since the determinant is then 
    \begin{align*}
      \mathrm{det}(\mathcal{L}_\mathcal{P}) &=\mathrm{det}( 2\cdot \CurA_n) \\ 
      &= 2^n \mathrm{det}(\CurA_n)= 2^n \sqrt{n+1}\\
      &= 
     2^{r-1}2^s\sqrt{n+1}.  
    \end{align*}
\end{remark}

Until now, we have restricted ourselves to the case where $\mathcal{P}$ 
only contains the ramified places of degree one and the inert places 
of degree 2. Here, we present one result on the case where 
$\mathcal{P}$ is chosen to be the set of all rational places.

\begin{theorem}\label{thm:Hyp_ell_ff_lat_det_rational_places}
    Let $F/K$ be a hyperelliptic function field over $\mathbb{F}_q$ 
    with genus $g$ and let $h$ denote the class number. 
    Take $\mathcal{P}$ to be the set of all rational 
    places of $F$, i.e. $\mathcal{P} = \mathbb{P}_F^1$. If $q$ and $g$
    satisfy 
    $$ \sqrt{q}+ \frac{1}{\sqrt{q}} > 2\cdot (2g-1),$$
    then the determinant of the function field lattice 
    $\mathcal{L}_\mathcal{P}$ is equal to 
    $$\mathrm{det}(\mathcal{L}_\mathcal{P}) = h\sqrt{n+1}.$$
\end{theorem}

To prove this theorem we first need the following result \cite[Theorem 34]{hess2004computing}.
Let $N_d(F/K)$ denote the number of places of degree one in the constant 
field extension of $F/K$ of degree $d$.

\begin{theorem}\label{thm:ClassGroup_generators}
    Let $D_0$ be a divisor of degree one in $F/K$ and $d\in \mathbb{N}$ 
    such that $N_d(F/K) > (g-1)2q^{\frac{d}{2}}$. Denote by $P_1,...,P_r$ all 
    places of degree dividing $d$ and set 
    $D_i = P_i - \mathrm{deg}(P_i)D_0$. Then $\mathrm{deg}(D_i) = 0$ 
    and the classes $[D_i]$ generate $\mathrm{Cl}^0(F/K)$.
    
\end{theorem}

Again, the goal here is to understand the group $\mathrm{Div}_{\mathcal{P}}^0(F) / \mathrm{Princ}_{\mathcal{P}}^0(F)$ and the order of its elements in order to determine $h_0$, when $\mathcal{P} = \mathbb{P}_F^1$. The theorem above helps us understand the conditions under which $\mathrm{Cl}^0(F/K)$ is generated by places of degree one. This, in turn, tells us when $h_0$ is equal to the class number $h$.

\begin{proof}[Proof of Theorem \ref{thm:Hyp_ell_ff_lat_det_rational_places}]
    We only consider the case where the place at infinity is ramified 
    in $F/K$ so we can choose $P_\infty$ as our divisor $D_0$. Next 
    recall by Hasse-Weil Theorem we have
    $\mid N_d(F/K) -(q^d-1) \mid \leq 2gq^\frac{d}{2}$. 
    If we take $d=1$ and use this bound it becomes clear when we fulfill 
    the necessary condition 
 %   \begin{align*}
 %       N_d(F/K) &\geq (q+1)-2g\sqrt{q} > (g-1)2\sqrt{q} \text{ if and only if}\\
 %       q+1 &> (2g-1)2\sqrt{q} \text{ if and only if}\\
 %       \sqrt{q}+\frac{1}{\sqrt{q}} &> 2(2g-1).
 %   \end{align*}
    $$ N_d(F/K) \geq (q+1)-2g\sqrt{q} > (g-1)2\sqrt{q} \text{ , that is,}$$
    $$ q+1 > (2g-1)2\sqrt{q} \text{ , or equivalently, }$$
    $$\sqrt{q}+\frac{1}{\sqrt{q}}> 2(2g-1).$$
    This means that when $g$ and $q$ satisfy the this inequality, then 
    by Theorem \ref{thm:ClassGroup_generators} the divisor classes $[P-P_\infty]$ for $P \in \mathbb{P}_F^1$ 
    generate $\mathrm{Cl}^0(F/K)$.
    This implies that every divisor class in $\mathrm{Cl}^0(F/K)$ has a 
    representative with support on $\mathcal{P} = \mathbb{P}_F^1$. 
    And hence, 
    $\mathrm{Div}_{\mathcal{P}}^0(F) / \mathrm{Princ}_{\mathcal{P}}^0(F)
    = \mathrm{Cl}^0(F/K)$. Therefore $h_0 = h$ and the determinant 
    of the function field lattice is $\mathrm{det}
    (\mathcal{L}_\mathcal{P}) = h\sqrt{n+1}$.
    
\end{proof}

\section{Automorphisms}\label{sec:Automorphisms}

In this section, we aim to understand the relationship between the automorphisms of the function field and the associated lattice. To achieve this, we discuss the explicit automorphism groups of specific function fields in the following subsections, addressing each case separately. Our focus is on the rational, elliptic, and hyperelliptic function fields.

\subsection{General Algebraic Function Field Lattices}\label{subsec:Aut of Lat from FF in general}

Let us first explain how we map automorphisms of the function field $F/K$ to automorphisms of the function field lattice $\mathcal{L}_\mathcal{P}$. For the induced map to be well-defined, it must be invariant on $\mathcal{O}_\mathcal{P}^*$. Otherwise, we would not be able to define the action of the map on a given lattice vector. This implies that any automorphism $\sigma \in \mathrm{Aut}_K(F)$ must be invariant on the set of places $\mathcal{P}$. In most cases considered in this chapter, $\mathcal{P} = \mathbb{P}_F^1$, which ensures that any automorphism will be invariant on $\mathcal{P}$.

For this reason, we first recall the following lemmas \cite{stichtenoth}.

\begin{lemma}\label{lemma:aut_bijection_places}
    Let $F/K$ be a function field and let $\sigma \in 
    \mathrm{Aut}_K(F)$. Then $\sigma$ induces a bijection on 
    the set of places $\mathbb{P}_F$.
\end{lemma}

 \begin{lemma}\label{lemma:aut_preserve_degree}
    Let $F/K$ be a function field and let $\sigma \in 
    \mathrm{Aut}_K(F)$. Then $\sigma$ preserves the degree 
    of a place. 
 \end{lemma}
Now recall the construction of a function field lattice from Definition \ref{defn:FuncFieldLattice}. The constructed lattice is isomorphic to the group of principal divisors of $F$ with support on $\mathcal{P}$. At the same time, we are not concerned with automorphisms that act trivially on the lattice. More specifically, these are maps that are the identity on $\mathcal{P}$. To address this, we introduce the following notation: let $F/K$ be a function field, and let $\mathcal{P}$ be a finite subset of $\mathbb{P}_F$. Define $\mathrm{Inv}(\mathcal{P})$ as the subgroup of all automorphisms of $F/K$ that are the identity on $\mathcal{P}$.

When $\mathcal{P}$ is the set of all places of a fixed degree, $\mathrm{Inv}(\mathcal{P})$ is a normal subgroup of $\mathrm{Aut}_K(F)$. We briefly explain why. Let $\mathcal{P}$ be the set of all places of degree $d \in \mathbb{N}$. Let $\sigma \in \mathrm{Inv}(\mathcal{P})$ and $\Tilde{\sigma} \in \mathrm{Aut}_K(F)$. As stated in Lemmas \ref{lemma:aut_bijection_places} and \ref{lemma:aut_preserve_degree} an automorphism acts as a bijection on the places of a given degree. This means that the map $\Tilde{ \sigma} \sigma \Tilde{ \sigma}^{-1} $ will permute the places of degree $d$ and then permute them back, so it remains an element of $\mathrm{Inv}(\mathcal{P})$.

Under the assumption that $\CurP$ is the set of places of a fixed degree, we define the following map: 

\begin{align}
    \Psi_\CurP : \mathrm{Aut}_K(F)/\mathrm{Inv}(\mathcal{P})
&\rightarrow 
    \mathrm{Aut}(\mathcal{L}_\mathcal{P}), \nonumber\\
    [\sigma] &\mapsto \Big{(} \Phi_\CurP(z) \mapsto \Phi_\CurP(\sigma(z))\Big{)}. \nonumber
\end{align}

\begin{theorem}\cite{bottcher2015lattices}
\label{thm: Gen_Aut_FF_Lat}
    Let $F/K$ be a function field over a finite field $\mathbb{F}_q$. Let $\mathcal{P} = \mathbb{P}_F^1$. Then 
    the map $\Psi_\CurP$ is an injective group homomorphism.
\end{theorem}

\begin{remark}
 For all lattices, including function field lattices, the map that multiplies all vectors by $-1$ is an automorphism. In the case of a function field lattice, multiplying all vectors by $-1$ corresponds to sending the underlying function $z \in F$ to its inverse $z^{-1}$, as this negates the valuations at any place, i.e. 
    $$v_P(z^{-1}) = -v_P(z).$$
    However, the map $z \mapsto z^{-1}$ is not an automorphism of the function field, as it is not a field homomorphism. This implies that the automorphism group of a function field and that of the associated lattice can never be isomorphic, unless both are trivial. In particular, this means that the images under our map $\Psi_\CurP$ will always have an index of at least 2. Additionally, the existence of the subgroup $\mathrm{Inv}(\CurP)$ implies that there are often automorphisms of the function field that do not define automorphisms on the lattice. As a result, the two automorphism groups are often not subgroups of each other. This is why, in the next section, we are motivated to observe them separately and then study their relationship.
\end{remark}

\subsection{Rational Function Field Lattices}\label{subsec:RationalFFLattices}

Recall that the automorphism group of the rational function field consists of maps in the form of M\"{o}bius transformations, that is, $x \mapsto \frac{ax+b}{cx+d}$, where $ad-bc\neq 0$. It is well-known that this group is explicitly the projective linear group
$$\mathrm{Aut}(K(x)/K)\simeq \mathrm{PGL}_2(K).$$
In this section, we claim that the automorphism group of the rational function field is a proper subgroup of the automorphism group of the lattice, denoted by $\mathrm{Aut}(\CurA_n)$. To support this, we first analyze $\Aut(\CurA_n)$ separately and then apply the results in Subsection \ref{sec:preliminaries}. The following lemma from \cite{martinet2013perfect} is the first step in our analysis.

\begin{lemma}\label{Lemma:AutOfRoot}
    For $n \geq 2$, $\mathrm{Aut}(\CurA_n)$ is isomorphic to $\{ \pm \Id \} \times S_{n+1}$. Here, $S_{n+1}$ denotes the permutation group given by permutations of the $n+1$ coordinates in $\mathbb{R}^{n+1}$ and
$\{ \pm \Id \}$ denotes the identity matrix and its additive inverse.
\end{lemma}

We can now discuss how the automorphism groups of the rational function field and the rational function field lattice are connected. The following theorem shows that they are indeed subgroups of each other. 
    
\begin{theorem}
    Let $K(x)/K$ be the rational function field. When 
    $\mathcal{P} = \mathbb{P}^1_F$ is the set of all rational places, then we have
    $\mathcal{L}_{\mathcal{P}} = \CurA_n$  where $|\mathcal{P}| = n+1$ 
    and  
    $$\mathrm{Aut}(K(x)) \subset \mathrm{Aut}(\CurA_n).$$
\end{theorem}

\begin{proof}
    By Theorem \ref{thm: Gen_Aut_FF_Lat} we know 
    that $\mathrm{Aut}(K(x))/\mathrm{Inv}(\mathcal{P})$ is a 
    subgroup of $\mathrm{Aut}(\CurA_n)$. Therefore, it remains to show that the subgroup 
    $\mathrm{Inv}(\mathcal{P})$ is trivial.\\
    Let $\sigma \in \mathrm{Inv}(\mathcal{P}),$
    where $\sigma$ is given by 
    $\sigma(x) = \frac{ax+b}{cx+d}$ with $a,b,c,d \in 
    K$ and $ad-bc\neq 0$. Since $\sigma$ is the identity on all rational places, it must, in particular, send $x$ to an element in $K(x)$ with the same divisor $P_x - P_\infty$. Because $\sigma(x) = \frac{ax+b}{cx+d}$, we can immediately conclude that $c$ and $b$ must be zero. This leaves us with $\sigma(x) = \frac{a}{d}x$. Since $\frac{a}{d}$ is a constant, this map is equivalent to $\sigma(x) = x$. Therefore, $\sigma$ is the identity map, and $\mathrm{Inv}(\mathcal{P})$ is trivial.

\end{proof}

\subsection{Lattices from Finite Abelian Groups}\label{subsec:LatticesFromAbelianGroups}

To understand the automorphisms of elliptic, hyperelliptic, and Hermitian function field lattices, we need the following construction from \cite{bottcher2015lattices}.

\begin{definition}
\label{defn:lat_from_abelian_groups_and_S}
    Let $G$ be an abelian group and $S= 
    \{ g_0 = 0, g_1, \ldots, g_{n} \}$ a subset of $G$. Then we define the lattice $\mathcal{L}_G(S)$ by
    $$ \mathcal{L}_G(S) \coloneqq \left\{ 
    (x_0, \ldots, x_{n}) \in \CurA_n \: \mid \: \sum_{i=0}^{n} x_i g_i = 0 \right\} .$$
    
\end{definition}

For these types of lattices, there are some useful results. We can induce automorphisms of the lattice from the automorphisms of the underlying group or set that defines its structure. This allows us to apply the automorphism groups of finite abelian groups, for which more is known. We will primarily use the following two lemmas.

\begin{lemma}\label{lemma:abelian_group_aut_1}
    Let $G$ be a finite abelian group and let $S$ be a subgroup of $G$. Then 
    $$\mathrm{Aut}(S) \cong \mathrm{Aut}(\mathcal{L}_G(S))
    \cap S_{n}.$$
\end{lemma}

In more general cases, $S$ is not a subgroup. In such cases, we need to consider the set of permutations on $S$ that extend to automorphisms of the entire group $G$. Let us denote this set by $\mathrm{Aut}(G,S)$. This set forms a subgroup of $\mathrm{Aut}(G)$.

\begin{lemma}\label{lemm:abelian_group_aut_generating_S}
    Let $G$ be a finite abelian group and $S$ a subset of $G$.
    Then $\mathrm{Aut}(G,S)$ is a subgroup of 
    $\mathrm{Aut}(\mathcal{L}_G(S)) \cap S_{n}$.\\
    Moreover, if $S$ is a generating set of $G$ then 
    $\mathrm{Aut}(G,S) \cong \mathrm{Aut}(\mathcal{L}_G(S))\cap S_{n}$.
\end{lemma}

\subsection{Elliptic Function Field Lattices}

In Subsection \ref{subsec:RationalFFLattices}, the automorphism group of the rational function field lattice is completely determined. However, finding all possible automorphisms of a lattice is generally not an easy task. For this reason, we focus on identifying large subgroups of the automorphism group instead. There are two main approaches to finding such large subgroups. The first approach is to note that all function field lattices are sublattices of $\CurA_n$, so we search for a subgroup of $\mathrm{Aut}(\CurA_n)$ that stabilizes a sublattice $\mathcal{L}$. This approach is used in \cite{martinet2013perfect} when studying the Barnes lattice $\CurB_n$. Essentially, it involves determining which permutations in $S_{n+1}$ still stabilize the lattice.. 

A different approach is taken in \cite{bottcher2015lattices} for lattices generated by finite abelian groups, as explained in Subsection \ref{subsec:LatticesFromAbelianGroups}. This approach considers that these lattices have rank $n$ and therefore focuses on subgroups of $S_n$. As we see in this section, elliptic function field lattices fall into this category of lattices.

Function field lattices always have size $n+1$ and rank $n$. The key difference between these two approaches lies in how one views the automorphisms: either as permutations of the coordinates of the lattice vectors or as permutations of the basis vectors.

We combine elements from both approaches to, hopefully, provide a more complete understanding of the automorphism group of these lattices.

Let us first examine the automorphisms of elliptic function fields themselves, closely following the results from \cite{ma2020group}.

\begin{theorem}\label{thm:ell_ff_aut_group}
    Let $E$ be an elliptic function field over $\mathbb{F}_q$. 
    Let $\mathrm{Aut}(E,\mathcal{O})$ be the subgroup of 
    automorphisms over $\mathbb{F}_q$ which stabilize 
    the place at infinity $\mathcal{O}$ and let $T_E$ be the 
    translation group of $E$. Then the automorphism group of 
    $E$ is given by 
    $$ \mathrm{Aut}(E) = T_E \rtimes \mathrm{Aut}(E,\mathcal{O}).$$
\end{theorem}

Naturally, $T_E$ has the structure of the group of rational places of the elliptic function field, where each place $Q$ induces an automorphism $\tau_Q(P) = P \oplus Q$. On the other hand, the subgroup $\mathrm{Aut}(E, \mathcal{O})$ is a subgroup of the automorphisms of the elliptic curve and can have at most $24$ elements. In this case, $\mathrm{Inv}(\mathcal{P})$ is trivial since all of these maps affect the rational places of $E$. However, these automorphisms do not cover the entire automorphism group of these lattices. To fully understand the group, it is more effective to approach the lattices directly.

Let $\mathcal{P} = {P_0, P_1, \ldots, P_n }$ be the set of rational places of $E$, where $P_0 = \mathcal{O}$ is the place at infinity. Recall that the set of rational places in an elliptic function field forms a group. As mentioned earlier in this chapter, lattices from elliptic function fields are special cases of lattices generated by abelian groups. 
They can equivalently be defined by 
$$ \mathcal{L}_\mathcal{P} = \left\{x \in \CurA_n \: | \: 
\sum_{i=0}^{n} [x_i] \bP_i = \mathcal{O} \right\},$$
where $\bP_i$ denotes the corresponding point of the rational place $P_i$.

Now we can apply Lemma \ref{lemma:abelian_group_aut_1} to our lattices from elliptic function fields.

\begin{theorem}\label{thm:ell_lat_aut_cap_sn}
    Let $\mathcal{P}$ be the group of rational places of 
    an elliptic function field $E$. Let 
    $\mathcal{L}_\CurP$ denote the lattice from 
    $\mathcal{O}_\mathcal{P}^*$. 
    Then $$\mathrm{Aut}(\mathcal{L}_\mathcal{P}) \cap 
    S_n \cong \mathrm{Aut}(\mathcal{P}).$$
\end{theorem}

Recall that the group of rational places is isomorphic to the 
group $\mathbb{Z}/n_1 \mathbb{Z} \times \mathbb{Z}/ 
n_2\mathbb{Z}$ where $n_1 | n_2$ and $n_1 \cdot n_2 = n+1$.

The automorphism groups of abelian groups have been studied in \cite{hillar2007automorphisms}. Below, we provide a brief note on how this applies to $\mathrm{Aut}(\mathcal{P})$ in the following remark.

\begin{remark}\label{rmk:ell_ff_aut_groupauts}
    When the group $\CurP$ is cyclic, i.e. $\mathcal{P} \cong 
    \mathbb{Z}/(n+1) \mathbb{Z}$, then the number of 
    automorphisms is simply $\varphi(n+1)$, where $\varphi$ is the Euler Totient function. When the group is not cyclic, we use the fact that for two 
    groups $G$ and $H$ with relatively prime order we have 
    $\mathrm{Aut}(G\times H)\cong \mathrm{Aut}(G)\times \mathrm{Aut}(H)$.
    Thus, we reduce this problem to the case  $\mathcal{P} \cong \mathbb{Z} / p^{e_1} \mathbb{Z} \times
    \mathbb{Z} / p^{e_2} \mathbb{Z}$ where $p$ is prime and  
    $e_1 \leq e_2$. 
    We have two cases, the order of the automorphism group 
    $\mathrm{Aut}(\mathbb{Z} / p^{e_1} \mathbb{Z} \times
    \mathbb{Z} / p^{e_2} \mathbb{Z})$ is  
    \begin{align*}
        \begin{cases}
            \varphi(p^{e_1})^2 \cdot p^{2e_1 -1} 
            \cdot (p+1) &\text{ for } e_1 = e_2,\\
            \varphi(p^{e_1})\cdot \varphi(p^{e_2})
            \cdot p^{2e_1 -1} \cdot p &\text{ for } e_1 \neq e_2.
        \end{cases}
    \end{align*}
\end{remark}

Note that all automorphisms from $\mathrm{Aut}(\CurP)$ fix the place at infinity, and therefore fix the first coordinate of the lattice vectors. This is because an automorphism of a group fixes the neutral element, which is why we intersect the automorphism group $\mathrm{Aut}(\mathcal{L}\CurP)$ with $S_n$ instead of $S{n+1}$. However, this limits the types of maps we can consider for the automorphism group. What is missing is the group of translations $T_E$, which, as discussed earlier in this section, are also automorphisms of the elliptic function field. Recall that this subgroup is isomorphic to the group of rational places $\mathcal{P}$.

\begin{lemma}\label{lemma:big_subgroup_ell_lat_auts}
    Let $\mathcal{L}_\mathcal{P}$ be an elliptic function field 
    lattice with $\mathcal{P} = \mathbb{P}_E^{(1)}$. Then the 
    automorphism group $\mathrm{Aut}(\mathcal{L}_\mathcal{P})$ 
    has a subgroup isomorphic to 
    $$ \{ \pm \Id \} \times (\mathcal{P} \rtimes 
    \mathrm{Aut}(\mathcal{P}) ),$$
    where $\rtimes$ denotes the semi-direct product of two groups.
    
\end{lemma}

\begin{proof}
   % We combine the two previous results, namely Theorem 
   % \ref{thm:ell_ff_aut_group} and Theorem \ref{thm:ell_lat_aut_cap_sn}. 
    We know 
    that by Theorem \ref{thm:ell_lat_aut_cap_sn} that $\mathrm{Aut}(\mathcal{P})$ is a subgroup of 
    $\mathrm{Aut}(\mathcal{L}_\mathcal{P})$. 
    By Theorem \ref{thm:ell_ff_aut_group} and Theorem \ref{thm: Gen_Aut_FF_Lat} we map over the group of translations $T_E$, which is isomorphic 
    to $\mathcal{P}$, to the automorphism group of the lattice. 
    This implies that $\mathcal{P}$ is also a 
    subgroup of $\mathrm{Aut}(\mathcal{L}_\mathcal{P})$. 
    It is easy to see that the intersection of these two 
    subgroups is only the identity, since all other 
    elements from 
    $\mathcal{P} \cong T_E$ shift the coordinate of the lattice 
    associated to $\mathcal{O}$ and all elements in 
    $\mathrm{Aut}(\mathcal{P})$ fix it. 
    
    Now we claim that $\mathcal{P}$ is normal in the 
    product $\mathcal{P}\cdot \mathrm{Aut}(\mathcal{P})$.
    Let $P,Q,R \in \mathcal{P}$ and $\sigma \in 
    \mathrm{Aut}(\mathcal{P})$ so that $\tau_Q \circ \sigma$ is 
    an element in the group $\mathcal{P}\cdot \mathrm{Aut}(\mathcal{P})$ and $\tau_P$ is an element in $\CurP \cong T_E$. Then 
    \begin{align*}
    (\sigma^{-1}\circ \tau_{-Q})\circ \tau_P\circ (\tau_Q\circ \sigma) (R) 
    &= \sigma^{-1}(\sigma(R)+Q+P-Q) \\
    &= \sigma^{-1}(\sigma(R)+P) \\
    &= R+ \sigma^{-1}(P)\\
    &= \tau_{\sigma^{-1}(P)}(R).  
    \end{align*}

    This proves that the product of these two subgroups is 
    isomorphic to their semidirect product. Hence 
    $\mathcal{P} \rtimes \mathrm{Aut}(\mathcal{P})$ is 
    isomorphic to a subgroup of $\mathrm{Aut}(\mathcal{L}_\mathcal{P})$.
\end{proof}

To the best of our knowledge, this provides a fairly large subgroup of the automorphism group of elliptic function field lattices. It is difficult to estimate how much larger the full group might be. However, in the specific case where $\mathcal{P}$ is cyclic, we can determine the full automorphism group. Indeed, when $\mathcal{P}$ is cyclic then this lattice is the 
Barnes lattice.
In \cite{martinet2013perfect} it has been proven that 
for $n \geq 11$ the subgroup in 
Lemma \ref{lemma:big_subgroup_ell_lat_auts} is the 
full automorphism group of the lattice and has order 
$2\cdot (n+1) \cdot \varphi(n+1)$ where $\varphi$ is the Euler Totient function.

\begin{remark}
\label{rmk:more_auts_with_dual}
    In this section we almost treat the automorphism group of the 
    Barnes lattice $\CurB_n$, like a subgroup of that of the root lattice $\CurA_n$. However, in some dimensions $\CurB_n$ has 
    automorphism which $\CurA_n$ does not have, i.e. it has more than 
    those we have shown so far. 
    This is also explained in \cite{martinet2013perfect}, that we 
    can use the dual of the lattice to perhaps find more automorphism of the lattice. By \cite[Theorem 5.3.7.]{martinet2013perfect} for $6\leq n \leq 10$ the dual of the Barnes lattice $\CurB_n^*$ has more 
    vectors of length $d(\CurA_n^*)$ than $\CurA_n^*$ does.
    
\end{remark}

\subsection{Hermitian Function Field Lattices}

In this section, we closely follow the results from \cite{bottcher2016lattices}. Hermitian function field lattices are also lattices derived from finite abelian groups. In this case, the group $G$ is the zero class group $\mathrm{Cl}^0(F)$, and $S$ is the 
set of divisor classes
\begin{align*}
   S &= \big\{ [P_i-P_\infty ] \mid 0\leq i \leq n \big\}, 
\end{align*} where the places $P_0 = P_\infty, P_1, \ldots, P_{n}$ are all the rational 
places in $H/K(x)$. Then, using the notation as in 
Definition \ref{defn:lat_from_abelian_groups_and_S}, the Hermitian function field lattice 
$\mathcal{L}_\mathcal{P}$ is isomorphic to 
\begin{align*}
 \mathcal{L}_{\mathrm{Cl}^0(F)}(S) \coloneqq 
\left\{ 
    (x_0, \ldots, x_{n}) \in \CurA_{n} \: \mid \: \sum_{i=1}^{n} x_i [P_i-P_\infty ] = [0] \in \mathrm{Cl}^0(F) \right\} .   
\end{align*}

In this case $S$ is a generating set of the zero class group, while $\CurP$ is the set of rational places in $H$. Therefore, by Lemma \ref{lemm:abelian_group_aut_generating_S} we conclude 
the following result on the automorphism group of the function field lattice. Recall that $\mathrm{Aut}(\mathrm{Cl}^0(F),S)$ is the group 
of permutations on $S$ that extend to automorphisms of $\mathrm{Cl}^0(F)$.

\begin{lemma}
    Let $H/K(x)$ be a hermitian function field. Let 
    $\mathcal{P}$ be the set of rational places of $H$. 
    Then $\mathrm{Aut}(\mathrm{Cl}^0(F),S) \cong \mathrm{Aut}(\mathcal{L}_\mathcal{P} )
    \cap S_{n}$.
\end{lemma}

\begin{remark}
As it is mentioned in \cite{bottcher2016lattices}, we know that the zero class  group of Hermitian function fields is isomorphic to the group $\mathbb{Z}_{q+1}^{q^2-q}$. 
Hence we know that $\mathrm{Aut}(\mathrm{Cl}^0(F),\mathcal{P})$, and therefore by extension $\mathrm{Aut}(\CurL_\CurP)\cap S_{n}$, is a subgroup of 
$\mathrm{Aut}(\mathbb{Z}_{q+1}^{q^2-q})$.
\end{remark}

Next, we check the connection between the automorphisms 
of the function field and those of the lattice. From Theorem \ref{thm: Gen_Aut_FF_Lat},
we know that $\mathrm{Aut}(H)/\mathrm{Inv}(\mathcal{P})$ is a 
subgroup of the automorphism group of the lattice. 
Fortunately, it is shown in \cite{bottcher2016lattices} that in the case of Hermitian function fields 
$\mathrm{Inv}(\mathcal{P})$ is trivial. 

\begin{theorem}
    $\mathrm{Aut}(H)$ is a subgroup of $\mathrm{Aut}(
    \mathcal{L}_\mathcal{P})$. 
    
\end{theorem}

Much is known on the automorphisms of Hermitian function fields, see \cite{stichtenoth}.

\subsection{Hyperelliptic Function Field Lattices}\label{subsec:Automorphism_Hyperelliptic_FF_Lat}

In this section, we represent our main results on automorphism groups of hyperelliptic function fields. 
%As we discuss in remark \Lilian{add ref to remakr}, when we 
%take $\mathcal{P}$ to be the set of ramified and inert places 
%of the hyperelliptic function field, the lattice is, 
%in some cases, similar to the root lattice $\CurA_n$. Since similar lattices 
%have isomorphic automorphism groups by Lemma \ref{lemma:similar_lat_iso_aut}, 
%this means we can apply the results for $\CurA_n$ to these cases.

\begin{lemma}\label{lemma:hyp_latt_similar_to_A}
    Let $H$ be a hyperelliptic function field with defining 
    equation $y^2 = f(x)$ for a square free polynomial $f(x)$. 
    Let $\CurP$ be the set of ramified places of degree $1$ 
    and inert places of of degree $2$ in $H$, with 
    $| \mathcal{P}| = n+1$. 
    If either $\mathrm{char}(H) = 2$, or 
    $\mathrm{char}(H) \neq 2$ and $f(x)$ does not split into 
    linear factors, then the hyperelliptic function field lattice 
    $\CurL_\CurP$ is similar to $\CurA_n$.
\end{lemma}

\begin{proof}
    By Theorem \ref{thm:Hyperll_ff_lat_basis_vs_well-rounded} we 
    know that when $\mathrm{char}(H) = 2$, or when $f(x)$ does not split into linear factors, that the lattice $\CurL_\CurP$ has a basis 
    of minimal vectors. On top of this, by Section \ref{eg:HyperellipticFuncFieldLattices} we know that these basis vectors and the minimal vectors of 
    $\CurL_\CurP$ are all vectors of the form 
    \begin{align*}
        (0, \ldots, 0, 2, 0, \ldots, 0, -2, 0, \ldots, 0)
    \end{align*}
    These vectors are all exactly $2$ times the minimal vectors of 
    the root lattice $\CurA_n$, which also form a basis of $\CurA_n$. 
    Hence there are basis matrices of these two lattices, which fulfill 
    the relation 
    \begin{align*}
        B_H = 2\cdot B_A.
    \end{align*}
    This implies that they are similar.
\end{proof}
This result together with the following lemma yields an immediate corollary on the automorphism group of hyperelliptic function field lattices.

\begin{lemma}\label{lemma:similar_lat_iso_aut}
    If two lattices are similar then their automorphism 
    groups are isomorphic.
\end{lemma}

\begin{corollary}
    Let $\CurL_\CurP$ be a hyperelliptic function field lattice as 
    defined in Lemma \ref{lemma:hyp_latt_similar_to_A}.
    Then the automorphism group of $\mathcal{L}_\mathcal{P}$ is 
    $$\mathrm{Aut}(\mathcal{L}_\mathcal{P}) 
    \cong \{ \pm \Id \} \times S_{n+1}.$$
\end{corollary}

\begin{proof}
   By Lemma \ref{lemma:hyp_latt_similar_to_A}, we know that the function field lattice is similar to 
    $\CurA_n$. Therefore, By Lemma \ref{lemma:similar_lat_iso_aut} means they have isomorphic automorphism groups.
    %We prove in Lemma \ref{lemma:hyp_latt_similar_to_A} that 
    %when $\mathrm{char}(H) = 2$ or when $f(x)$ does not split, 
    %that the function field lattice is similar to 
    %$\CurA_n$. 
   % By Lemma \ref{lemma:similar_lat_iso_aut} means they have isomorphic 
   % automorphism groups.
\end{proof}

We note that this result only covers a specific case. 
Next we analyze what happens when $f(x)$ splits into 
linear factors. Recall that by Subsection \ref{eg:HyperellipticFuncFieldLattices} this case in particular is 
interesting because the lattice is generated by 
the vectors of the form $\Phi_\CurP(2P_i-2P_j)$, which have length 
$2$, and one additional vector $\Phi_\CurP(u)$, which has length 
$\sqrt{2g+2}$. The following theorem addresses this case.

\begin{theorem}\label{thm:Auts_Hyp_FF_Lat_fsplits}
    Let $H$ be a hyperelliptic function field with defining 
    equation $y^2=f(x)$, where $f(x)$ is a square-free polynomial of degree $2g+1$, which splits into linear 
    factors. Let $\mathcal{P}$ be the set of ramified places and inert places of of degree $2$ in $H$, with 
    $|\mathcal{P}| = n+1$.
    Then the automorphism group of the function field lattice $\mathrm{Aut}(\mathcal{L}_\mathcal{P})$ 
    has a subgroup isomorphic to 
    $$ S_{2g+2} \times S_s $$
    where $2g+2$ is the number of ramified places and $s$ is the number of inert places of degree $2$.
\end{theorem}

\begin{proof}
    For this proof we are fundamentally splitting the 
    lattice in two. First we tackle the $s$ coordinates 
    associated with the inert places of degree $2$. 
    Remember that all function field lattices are sublattices 
    of the root lattice $\CurA_n$, which has automorphism group 
    $\{ \pm \Id \} \times S_{n+1}$. Then we can focus on the 
    subgroup $S_s \subset S_{n+1}$ which only permutes the 
    last $s$ coordinates of the vector. We claim that these 
    automorphisms of $\CurA_n$ all define automorphisms on $\mathcal{L}_\mathcal{P}$.

    Let $v \in \mathcal{L}_\mathcal{P}$ and let 
    $\sigma \in S_s$ be an automorphism of $\CurA_n$ which only 
    permutes the last $s$ places. This means $\sigma$ 
    defines a 
    bijection on the places $Q_1, \ldots , Q_s$. We just need to 
    check that $\sigma(v)$ is still in the lattice 
    $\mathcal{L}_\mathcal{P}$. For convenience we 
    will work with the equivalent principle divisor, 
    instead of the vector $v$, 
    $$D = \sum_{i=0}^{2g+1} a_i P_i + 
    \sum_{j=1}^{s} b_j Q_j.$$  
    Then the image of this divisor under $\sigma$ is 
    \begin{align*}
      \sigma(D) &=  \sum_{i=0}^{2g+1} a_i P_i + \sum_{j=1}^{s} 
    b_{\sigma(j)} Q_{\sigma(j)}\\ &= \sum_{i=0}^{2g+1} a_i P_i + \sum_{j=1}^{s} b_j Q_j 
    + \sum_{j=1}^{s} b_{\sigma(j)} (Q_{\sigma(j)}-Q_j).  
    \end{align*}
    
    Because the places $Q_i$ and $Q_j$ are inert and of the same degree, the divisor 
    $Q_i-Q_j$ is the conorm of a principal divisor in $K(x)$ and therefore is principal itself. This implies that  
    $\sigma(D)$ is really equal to 
    $$ \sigma(D) = D + \Tilde{D}$$
    where $\Tilde{D}$ is a principal divisor with support on 
    $\mathcal{P}$. This means $\sigma(D)$ is again a 
    principal divisor with support on $\mathcal{P}$ and 
    therefore a vector of the lattice 
    $\mathcal{L}_\mathcal{P}$.
    This is enough to show that $\sigma$, when 
    restricted to $\mathcal{L}_\mathcal{P}$, is an 
    automorphism.

    Next we focus on the ramified places. Hyperelliptic function field lattices can also be viewed as a type of lattice generated by abelian groups. The natural choice for the group from 
    Definition \ref{defn:lat_from_abelian_groups_and_S} is 
    $\mathrm{Cl}^0(F)$ and we take the set $S$ to be the set 
    of divisor classes $[P_i-P_\infty ]$ for $1\leq i \leq 2g+1$ and 
    $[Q_j- 2 P_\infty ]$ for $1\leq j \leq s$. Then the 
    lattice $\mathcal{L}_\mathcal{P}$ is equivalent to
    
    \begin{align*}
       \mathcal{L}_{\mathrm{Cl}^0(F)}(S) = \Biggl\{
    (a_0, ...&, a_{2g+1}, b_1, ... , b_s) \in A_n \: \mid \: 
    \\
      &\sum_{i=1}^{2g+1} a_i [P_i-P_\infty ]  + \sum_{j=1}^{s} b_j [Q_j- 2 P_\infty ] 
    = [0] \in \mathrm{Cl}^0(F) \Biggr\} . 
    \end{align*}
 
    Again, recall that all of the divisors $Q_j- 2 P_\infty $ are 
    already principal, since they are the conorms of principal divisors in $K(x)$, so the above equation simplifies to 
    \begin{align*}
        \mathcal{L}_{\mathrm{Cl}^0(F)}(S) = 
    \left\{ 
    (a_0, ..., a_{2g+1}, b_1, ... , b_s) \in A_n \: \mid \: 
    \sum_{i=1}^{2g+1} a_i [P_i-P_\infty ]  
    = [0] \in \mathrm{Cl}^0(F) \right\}.
    \end{align*}
    
    Even though $\mathrm{Cl}^0(F)$ seems like the natural choice here, we actually get the exact same lattice 
    if we just use the subgroup generated by the elements 
    $[P_i-P_\infty]$. From \cite{stichtenoth} we know that this group 
    is isomorphic to $(\mathbb{Z}/2\mathbb{Z})^{2g}$, since these divisors all have order two and they fulfill the relation $\sum_{i=1}^{2g+1} [P_i-
    P_\infty ] = [0]$. Therefore 
    this lattice is equivalent to 

    \begin{align*}
        \mathcal{L}_{(\mathbb{Z}/2\mathbb{Z})^{2g}}(S) = 
    \left\{ 
    (a_0, ..., a_{2g+1}, b_1, ... , b_s) \in A_n \: \mid \: 
    \sum_{i=1}^{2g+1} a_i [P_i-P_\infty ]  
    = [0] \in (\mathbb{Z}/2\mathbb{Z})^{2g} \right\}.
    \end{align*}

    Now we can apply Lemma \ref{lemm:abelian_group_aut_generating_S}
    to this lattice. We need to 
    check $\mathrm{Aut}((\mathbb{Z}/2\mathbb{Z})^{2g},S)$, i.e. the permutations 
    of $S$ which extend to automorphisms of $(\mathbb{Z}/2\mathbb{Z})^{2g}$. 
    We claim that all permutations of $S$ can be extended to 
    automorphisms of this group.

    %From the paper on automorphism groups of finite abelian 
    %groups \cite{hillar2007automorphisms} we know that the 
    %utomorphism group of $(\mathbb{Z}/2\mathbb{Z})^{2g}$ has 
    %cardinality $\prod_{k=1}^{2g}(2^{2g}-2^{k-1})$. 

    Let $\sigma$ be a permutation in $S_{2g+1}$. We can view 
    $(\mathbb{Z}/2\mathbb{Z})^{2g}$ as a $2g$ dimensional 
    vector space over $ \mathbb{Z}/2\mathbb{Z}$. Then define the map 
    $\varphi : (\mathbb{Z}/2\mathbb{Z})^{2g} \rightarrow 
    (\mathbb{Z}/2\mathbb{Z})^{2g}$ by sending the generating elements 
    $ P_1, ... , P_{2g}$ to
    $ P_{\sigma(i)}, ... , P_{\sigma(2g)}$ 
    respectively. Then the last place $P_{2g+1} = \sum_{i=1}^{2g} P_i$ 
    is mapped to $\sigma(P_{2g+1}) = $
\begin{align*}
    \varphi(P_{2g+1}) = \varphi(\sum_{i=1}^{2g} P_i) = 
    \sum_{i=1}^{2g} P_\sigma(i) = P_{\sigma(2g+1)}.
\end{align*}
    
    Since $\sigma$ is a bijection the map 
    $\varphi$ is a bijective homomorphism and therefore an automorphism of the group.
    This shows that $S_{2g+1}$ is a subgroup of the 
    automorphism group of the lattice.

    Then we can explicitly construct a transposition of the coordinates of the place at infinity $P_\infty$ and another ramified place $P_1$. To this end, define the map $\tau$ on 
    $\R^{2g+2+s}$ to be the linear map which permutes the 
    first two coordinates of the vectors. Since this is already
    an isometry, we only need to show that is sends lattice 
    vectors to lattice vectors. Again let us consider principal 
    divisors instead of vectors for convenience. 
    Let $D = a_0 P_\infty + \sum_{i=1}^{2g+1}a_i P_i +  \sum_{j=1}^{s}b_i Q_i$ be a principal divisor of $F$. 
    Then $\tau(D)$ is equal to 
    \begin{align*}
        \tau(D) &= a_1 P_\infty + a_0 P_1 +\sum_{i=2}^{2g+1}a_i P_i +  \sum_{j=1}^{s}b_i Q_i \\
        &= D + [(a_1-a_0)P_\infty + (a_0-a_1)P_1].
    \end{align*}
    Since $D$ is principal and the places $P_\infty$, $P_1$ are ramified, we 
    know that the integers $a_0$ and $a_2$ are either both even or 
    both odd. This follows from how these divisors are generated, see 
    Subsection \ref{eg:HyperellipticFuncFieldLattices}. 
    Hence $(a_1-a_0)$ and $(a_0-a_1)$ are both even. 
    Therefore the divisor $(a_1-a_0)P_\infty + (a_0-a_1)P_1$ is 
    principal since it is the conorm of the 
    divisor $\frac{1}{2}(a_1-a_0)P_\infty + \frac{1}{2}(a_0-a_1)P_1$ in $K(x)$.
    This implies that $\tau(D)$ is the sum of two principal divisors, 
    making it also principal. We conclude that $\tau$ defines a 
    lattice automorphism of $\mathcal{L}_\mathcal{P}$.
    
    Together, $\tau$ and $S_{2g+1}$ generate a subgroup of the 
    automorphism group which is isomorphic to 
    $S_{2g+2}$. In addition, it is easy to check that the two 
    subgroups $S_{2g+2}$ and $S_s$ intersect trivially.

\end{proof} 

\begin{remark}
    Unlike in the first case, any permutation of a ramified and inert 
    place will not define an automorphism on the lattice.  Due to the inclusion of vectors like $\Phi_\CurP(u)$, it becomes clear that entries in coordinates corresponding to inert places are always even, while those corresponding to ramified places are not. Permuting these would result in a map that is not a bijection on the lattice.
\end{remark}

Let us now take a closer look at the case where the genus $g$ equals 3. As mentioned in previous sections, lattices that are more symmetric or have large kissing numbers typically have larger automorphism groups.

When $g=3$ and the defining polynomial $f(x)$ splits into linear factors, all the generating vectors have the same length. In this case, the function field lattice has a much larger kissing number compared to other lattices from hyperelliptic function fields.

Fortunately, $g=3$ is small enough to be easily implemented by hand in Magma, making it straightforward to check the exact automorphism group.

\begin{example}
    Let $H$ be the hyperelliptic function field over $\mathbb{F}_{11}$, 
    with defining equation $y^2=f(x)$, where
    \begin{align*}
      f(x)= x^7+5x^6+3x^5+9x^4+4x^3+2x^2+9.  
    \end{align*}
    
    Then $H$ has genus $g=3$ and has 
    $8$ ramified places of degree $1$, which are the roots of $f(x)$ and 
    the place at infinity, and two inert places of degree $2$.
    The resulting lattice is of length $10$ and rank $9$ with basis 
    matrix
    $$B = \begin{pmatrix}
        -7&1&1&1&1&1&1&1&0&0\\
        -2&0&2&0&0&0&0&0&0&0\\
        -2&0&0&2&0&0&0&0&0&0\\
        -2&0&0&0&2&0&0&0&0&0\\
        -2&0&0&0&0&2&0&0&0&0\\
        -2&0&0&0&0&0&2&0&0&0\\
        -2&0&0&0&0&0&0&2&0&0\\
        -2&0&0&0&0&0&0&0&2&0\\
        -2&0&0&0&0&0&0&0&0&2
    \end{pmatrix}$$
    With a given basis matrix it is easy to implement the lattice in 
    magma and compute the size of its automorphism group.

    %\begin{lstlisting}[frame = single]
    %> L := Lattice(10, 
    % [-7,1,1,1,1,1,1,1,0,0, -2,0,2,0,0,0,0,0,0,0,
    %-2,0,0,2,0,0,0,0,0,0, -2,0,0,0,2,0,0,0,0,0, 
    %-2,0,0,0,0,2,0,0,0,0, -2,0,0,0,0,0,2,0,0,0, 
    %-2,0,0,0,0,0,0,2,0,0, -2,0,0,0,0,0,0,0,2,0, 
    %-2,0,0,0,0,0,0,0,0,2]);
    %> G := AutomorphismGroup(L);
    %> FactoredOrder(G);
    %[ <2, 9>, <3, 2>, <5, 1>, <7, 1> ]
    %\end{lstlisting}

    \begin{Verbatim}[frame=single, tabsize=5]
    > L := Lattice(10, [-7,1,1,1,1,1,1,1,0,0, -2,0,2,0,0,0,0,0,0,0,
    -2,0,0,2,0,0,0,0,0,0, -2,0,0,0,2,0,0,0,0,0, -2,0,0,0,0,2,0,0,0,
    0, -2,0,0,0,0,0,2,0,0,0, -2,0,0,0,0,0,0,2,0,0, -2,0,0,0,0,0,0,
    0,2,0, -2,0,0,0,0,0,0,0,0,2]);
    > G := AutomorphismGroup(L);
    > FactoredOrder(G);
    [ <2, 9>, <3, 2>, <5, 1>, <7, 1> ]
	\end{Verbatim}

    The result is that this lattice has an automorphism group of size 
    $2\cdot 8! \cdot 2!$. By Theorem \ref{thm:Auts_Hyp_FF_Lat_fsplits} the group 
    is therefore isomorphic to $\{ \pm \Id \} \times S_8 \times S_2$.
    This is rather surprising, since we would expect that, if any of the 
    hyperelliptic function field lattices had a larger automorphism 
    group, it would be this one. 
    
\end{example}

\begin{center}
    \subsection*{Acknowledgments}
\end{center}
We acknowledge Renate Scheidler for her guidance on semi-reduced divisors in algebraic function fields, which made an important contribution to the results in Section \ref{sec:Determinant}.

%\section{References}
% Place the content of the .bbl file directly here

% Example:

%\end{thebibliography}

\end{document}